\documentclass{amsart}
\usepackage{amsmath, amsthm, amsfonts, amssymb}
\usepackage{mathrsfs}

\newtheorem{theorem}{Theorem}[section]
\newtheorem{lemma}{Lemma}[section]

\newtheorem{definition}{Definition}[section]

\theoremstyle{remark}
\newtheorem{remark}{Remark}[section]

\numberwithin{equation}{section}

\begin{document}

\title[WEIGHTED LIPSCHITZ ESTIMATE FOR COMMUTATORS]
{WEIGHTED LIPSCHITZ ESTIMATE FOR COMMUTATORS OF ONE-SIDED OPERATORS
ON ONE-SIDED TRIEBEL-LIZORKIN SPACES}


\author[Zun Wei Fu,\quad Qing Yan Wu,\quad Guang Lan Wang]
{Zun Wei Fu,\quad Qing Yan Wu$^*$,\quad Guang Lan Wang}
\address{School of Sciences,  Linyi University, Linyi, 276005, P.R. China}

\email{lyfzw@tom.com,\quad wuqingyan@lyu.edu.cn,\quad wangguanglan@lyu.edu.cn}
\thanks{\noindent$^*$ Corresponding author}

\thanks{This work was partially supported by NSF of China (Grant Nos. 10901076, 11171345 and 11126203), NSF of
Shandong Province (Grant Nos. ZR2010AL006).}



\subjclass[2010]{42B20, 42B25}

\keywords{one-sided weight, weighted Lipschitz function, one-sided Triebel-Lizorkin space, one-sided singular integral, one-sided discrete square function}

\begin{abstract}
Using the extrapolation of one-sided weights,
 we establish the boundedness of commutators generated by weighted Lipschitz functions and  one-sided singular integral operators
 from weighted Lebesgue spaces to weighted one-sided Triebel-Lizorkin spaces. The corresponding results for
 commutators of one-sided discrete square functions are also obtained.
\end{abstract}

\maketitle

\section{Introduction}
The study of one-sided operators was motivated not only as the generalization of the theory of both-sided
ones but also by the requirement in ergodic theory. In \cite{Sa}, Sawyer studied the weighted theory of
one-sided maximal Hardy-Littlewood operators in depth for the first time. Since then, numerous papers
have appeared, among which we choose to refer to \cite{AF}, \cite{AS},
\cite{FLSS}, \cite{MOT}, \cite{MR}, \cite{MT1} about one-sided operators, \cite{AC}, \cite{MPT}, \cite{MT2}, \cite{RS}
about one-sided spaces
and so on. Interestingly, lots of results show that for a class of smaller operators (one-sided operators)
and a class of wider
weights (one-sided weights), many famous results in harmonic analysis still hold.

Recently, Lorente and Riveros introduced the commutators of one-sided operators. In \cite{LR1}, they investigated
the weighted boundedness for commutators generated by several one-sided operators (one-sided discrete square functions,
 one-sided fractional operators, one-sided maximal operators of a certain type) and BMO functions. Recall that
 a locally integrable function $f$ is said to belong to $BMO(\mathbb{R})$ if
 $$\|f\|_{BMO}=\sup_I\frac{1}{|I|}\int_I|f-f_I|<\infty,$$
 where $I$ denotes any bounded interval and $f_I=\frac{1}{|I|}\int_{I}f(y)dy$.
In \cite{LR2}, \cite{LR3}, they obtained the weighted inequalities for commutators of a certain kind of one-sided operators
(one-sided singular integrals and other one-sided operators appeared in \cite{LR1}) and the weighted BMO functions.

Very recently, Fu and Lu \cite{FL} introduced a class of one-sided Triebel-Lizorkin spaces and studied the boundedness for commutators
(with symbol $b\in \rm{Lip}_\alpha$) of one-sided Caler\'on-Zygmund singular integral operators and one-sided fractional
integral operators. A function $b\in \rm{Lip}_\alpha$, $0<\alpha<1$, if it satisfies
$$\|b\|_{\rm{Lip}_\alpha}=\sup_{x,h\in\mathbb{R},\,h\neq 0}\frac{|b(x+h)-b(x)|}{|h|^{\alpha}}<\infty.$$
And it has the following equivalent form \cite{Pa}:
$$\|f\|_{\rm{Lip}_\alpha}\approx\sup_{I}\frac{1}{|I|^{1+\alpha}}\int_I|f-f_I|\approx
\sup_{I}\frac{1}{|I|^\alpha}\left(\int_I|f-f_I|^q\right)^{\frac{1}{q}},$$
where $1\leq q<\infty$. Obviously, if $\alpha=0$, then $f\in BMO$. In fact, BMO and $\rm{Lip}_\alpha$ are the special
cases of Campanato spaces (cf. \cite{Pe1}).
It should be noted that just like functions in BMO may be unbounded, such as $\log|x|$. The functions in
$\rm{Lip}_\alpha$ are not necessarily bounded either, for example $|x|^\alpha\in\rm{Lip}_\alpha$. Therefore, it is
also meaningful to investigate the commutators generated by operators and Lipschitz functions (cf. \cite{HG}, \cite{LLP}, \cite{Pa}).

Inspired by the above results, we concentrate on the boundedness for commutators (with symbol $b$ belonging to weighted Lipschitz spaces) of one-sided
singular operators as well as one-sided discrete square functions
from weighted Lebesgue spaces to weighted one-sided Triebel-Lizorkin
spaces.

In \cite{Pa}, Paluszy\'nski introduced a kind of Triebel-Lizorkin spaces $\dot{F}^{\alpha,
\infty}_{p}$. Fu and Lu \cite{FL} gave their one-sided versions.

\begin{definition} \cite{FL}
For $0<\alpha<1$ and $1<p<\infty$, one-sided Triebel-Lizorkin spaces $\dot{F}^{\alpha,
\infty}_{p, +}$ and $\dot{F}^{\alpha, \infty}_{p,
-}$ are defined by
$$\|f\|_{\dot{F}^{\alpha,
\infty}_{p,
+}}\approx\left\|\sup_{h>0}\frac{1}{h^{1+\alpha}}\int_{x}^{x+h}|f-f_{[x,
x+h]}|\right\|_{L^{p}}<\infty,$$ ºÍ
$$\|f\|_{\dot{F}^{\alpha,
\infty}_{p,
-}}\approx\left\|\sup_{h>0}\frac{1}{h^{1+\alpha}}\int_{x-h}^{x}|f-f_{[x-h,
x]}|\right\|_{L^{p}}<\infty,$$
where $f_{[x,
x+h]}=\frac{1}{h}\int_{x}^{x+h}f(y)dy$.
\end{definition}

\begin{remark}
It is clear that $\dot{F}^{\alpha,
\infty}_{p}\subsetneqq \dot{F}^{\alpha, \infty}_{p,+},$
$\dot{F}^{\alpha, \infty}_{p}\subsetneqq \dot{F}^{\alpha,
\infty}_{p,-}$ ÇÒ $\dot{F}^{\alpha, \infty}_{p,
+}\cap\dot{F}^{\alpha, \infty}_{p, -}=\dot{F}^{\alpha, \infty}_{p}$.
\end{remark}

Furthermore, the weighted one-sided Triebel-Lizorkin spaces have been defined in \cite{FL}.
\begin{definition}
For $0<\alpha<1$, $1<p<\infty$ and an appropriate weight $\omega$,
the weighted one-sided Triebel-Lizorkin spaces $\dot{F}^{\alpha,
\infty}_{p, +}$ and $\dot{F}^{\alpha, \infty}_{p,
-}$ are defined by
$$\|f\|_{\dot{F}^{\alpha,
\infty}_{p,
+}(\omega)}\approx\left\|\sup_{h>0}\frac{1}{h^{1+\alpha}}\int_{x}^{x+h}|f-f_{[x,
x+h]}|\right\|_{L^{p}(\omega)}<\infty,$$
and
$$\|f\|_{\dot{F}^{\alpha,
\infty}_{p,
-}(\omega)}\approx\left\|\sup_{h>0}\frac{1}{h^{1+\alpha}}\int_{x-h}^{x}|f-f_{[x-h,
x]}|\right\|_{L^{p}(\omega)}<\infty.$$
\end{definition}

One of the main objects of our study is the one-sided singular integral operators.
Assume that $K\in L^1(\mathbb{R}\setminus\{0\})$, $K$ is said to be a Calder\'on-Zygmund kernel
if the following properties are satisfied:\\
(a) There exists a positive constant $B_1$ such that
$$\left|\int_{\varepsilon<|x|<N}K(x)dx\right|\leq B_{1},$$
for all $\varepsilon$ and $N$ with $0<\varepsilon<N$, and the limit
$\lim_{\varepsilon\rightarrow0^{+}}\int_{\varepsilon<|x|<1}K(x)dx$ exists.\\
(b)  There exists a positive constant $B_{2}$ such that
$$|K(x)|\leq \frac{B_{2}}{|x|},$$
for all $x\neq0.$\\
(c) There exists a positive constant $B_{3}$ such that
$$|K(x-y)-K(x)|\leq \frac{B_{3}|y|}{|x|^{2}},$$
for any $x$ and $y$ with $|x|>2|y|$.

The singular integral with Calder\'on-Zygmund kernel $K$ is defined by
$$Tf(x)=\mathrm{p.v.}\int_{\mathbb{R}}K(x-y)f(y)dy.$$

\begin{definition} \cite{AF}
A one-sided singular integral $T^+$ is a singular integral associated to a
 Calder\'on-Zygmund kernel $K$ with support in $(-\infty, 0)$:
$$T^{+}f(x)=\lim_{\varepsilon\rightarrow0^{+}}\int_{x+\varepsilon}^{\infty}K(x-y)f(y)dy.$$
Similarly, when the support of $K$ is in $(0,+\infty)$,
$$T^{-}f(x)=\lim_{\varepsilon\rightarrow0^{+}}\int_{-\infty}^{x-\varepsilon}K(x-y)f(y)dy.$$
\end{definition}

The other main object in this paper is the one-sided discrete square function.
As is known, discrete square function is of interest in ergodic theory and has been extensively studied (cf. \cite{JKGW}).
\begin{definition}
The one-sided discrete square function $S^+$ is defined by
$$S^{+}f(x)=\left(\sum_{n\in \mathbb{Z}}\left|A_nf(x)-A_{n-1}f(x)\right|^2\right)^{\frac{1}{2}}.$$
for locally integrable $f$, where $A_nf(x)=\frac{1}{2^n}\int_{x}^{x+2^n}f(y)dy$.
\end{definition}
 It is easy to see that $S^+f(x)=\|U^+f(x)\|_{l^2}$, where $U^+$ is the sequence valued operator
\begin{equation}
U^+f(x)=\int_{\mathbb{R}}H(x-y)f(y)dy, \label{U}
\end{equation}
 here
 $$H(x)=\left\{\frac{1}{2^n}\chi_{(-2^n,0)}(x)-\frac{1}{2^{n-1}}\chi_{(-2^{n-1},0)}(x)\right\}_{n\in\mathbb{Z}}.$$
(see \cite{TT}).

\begin{definition}
The one-sided Hardy-Littlewood maximal operators $M^+$ and $M^-$ are defined by
$$M^{+}f(x)=\sup_{h>0}\frac{1}{h}\int_{x}^{x+h}|f(y)|dy,$$
and
$$M^{-}f(x)=\sup_{h>0}\frac{1}{h}\int_{x-h}^{x}|f(y)|dy,$$
for locally integrable $f$.
\end{definition}

The good weights for these operators are one-sided weights. Sawyer \cite{Sa} introduced the one-sided $A_p$ classes
$A_p^{+},~A_p^{-}$, which are defined by the
following conditions:
$$
A_{p}^{+}:\quad A_{p}^{+}(w):=\sup_{a<b<c}\frac{1}{(c-a)^{p}}\int_{a}^{b}w(x)\,dx\left(\int_{b}^{c}w(x)^{1-p'}\,dx\right)^{p-1}<\infty,
$$
$$
A_{p}^{-}:\quad A_{p}^{-}(w):=\sup_{a<b<c}\frac{1}{(c-a)^{p}}\int_{b}^{c}w(x)\,dx\left(\int_{a}^{b}w(x)^{1-p'}\,dx\right)^{p-1}<\infty,
$$
 when $1<p<\infty$;  also, for $p=1$,
$$
A_{1}^{+}:\quad M^{-}w(x)\leq Cw(x),\quad a.e.,
$$
$$
A_{1}^{-}:\quad M^{+}w(x)\leq Cw(x),\quad a.e..
$$

Let's recall the definition of weighted Lipschitz spaces given in \cite{HG}.
\begin{definition}
For $f\in L_{loc}(\mathbb{R})$, $\mu\in A^{\infty}$, $1\leq p\leq\infty$, $0<\beta<1$, we say that $f$ belongs to the weighted
Lipschitz space $Lip_{\beta,\mu}^p$ if
$$\|f\|_{ Lip_{\beta,\mu}^p}=\sup_{I}\frac{1}{\mu(I)^{\beta}}\left[\frac{1}{\mu(I)}
\int_{I}|f(x)-f_I|^p\mu(x)^{1-p}dx\right]^{\frac{1}{p}}<\infty, $$
where $I$ denotes any bounded interval and $f_I=\frac{1}{|I|}\int_If$.
\end{definition}

The weighted
Lipschitz space $Lip_{\beta,\mu}^p$ is a Banach space (modulo constants). Set $Lip_{\beta,\mu}=Lip_{\beta,\mu}^1$, By \cite{GC}, when
$\mu\in A_1$, then the spaces $Lip_{\beta,\mu}^p$ coincide, and the norms $\|\cdot\|_{ Lip_{\beta,\mu}^p}$ are equivalent
 for different $p$ with $1\leq p\leq\infty$, thus $\|\cdot\|_{ Lip_{\beta,\mu}^p}\sim\|\cdot\|_{ Lip_{\beta,\mu}}$ for any
 $1\leq p\leq\infty$.
It is clear that for $\mu\equiv1$, the space $Lip_{\beta,\mu}$ is the classical Lipschitz space $Lip_{\beta}$. Therefore,
weighted
Lipschitz spaces are generalizations of the classical Lipschitz spaces.

\begin{definition} \cite{LR1}
For appropriate $b$, the commutators of $T^+$ and $S^+$ are defined by
$$T^{+}_{b}f(x)=\int_{x}^{\infty}(b(x)-b(y))K(x-y)f(y)dy,$$
and
$$S^{+}_{b}f(x)=\left\|\int_{\mathbb{R}}(b(x)-b(y))H(x-y)f(y)dy\right\|_{l^2},$$
respectively.\end{definition}

Now, we formulate our main results as follows.

\begin{theorem}\label{th:1}
Assume that $1<p<\infty$, $v\in A_p$ and $w\in A_p^{+}$ are such that
$\mu^{1+\alpha}=(\frac{v}{w})^{\frac{1}{p}}$ for some $0<\alpha<1$
and $\mu\in A_1$. Then, for $b\in Lip_{\beta,\mu}$, there exists $C>0$ such that
$$\|T^{+}_bf\|_{\dot{F}^{\alpha,\infty}_{p,+}(w)}\leq C\|f\|_{L^p(v)},$$
for all bounded $f$ with compact support.
\end{theorem}

\begin{theorem}\label{th:2}
Assume that $1<p<\infty$, $v\in A_p$ and $w\in A_p^{+}$ are such that
$\mu^{1+\alpha}=(\frac{v}{w})^{\frac{1}{p}}$ for some $0<\alpha<1-\frac{1}{1+\varepsilon(w,v)}$
and $\mu\in A_1$.
Then, for $b\in Lip_{\beta,\mu}$, there exists $C>0$ such that
$$\|S^{+}_bf\|_{\dot{F}^{\alpha,\infty}_{p,+}(w)}\leq C\|f\|_{L^p(v)},$$
for all bounded $f$ with compact support.
\end{theorem}

\begin{remark}
In Theorem \ref{th:2}, $\varepsilon(w,v)$ is a positive number depending only on $w,v$. Since the condition that
is satisfied by $H$ in \eqref{U} is weaker than that of
Calder\'on-Zygmund kernel $K$ (see \cite{TT}). Naturally, the requirement of $\alpha$ in Theorem \ref{th:2} should be stronger.
\end{remark}

We remark that although like \cite{LR2}, \cite{LR3}, we will continue to use the one-sided
maximal functions to control the commutators of the two operators in this paper. The difference is that,
 by definition of one-sided Tiebel-Lizorkin spaces, the proof in this paper goes without using of one-sided
 sharp maximal operators.

In Section 2, we will give some necessary lemmas.
Then we will prove Theorem \ref{th:1} in Section 3. In the last section, we will give the proof of Theorem \ref{th:2}.
Throughout this paper the letter $C$ will be used to denote various constants, and the various uses of the letter
do not, however, denote the same constant.

\section{Preliminaries}
In order to prove our results, we will firstly introduce some necessary lemmas.
\begin{lemma}\cite{MPT}\label{th:lem1}
Suppose that $\omega\in A_1^{-}$, then there exists $\varepsilon_1>0$ such that for all $1<r\leq 1+\varepsilon_1$,
$w^r\in A_1^{-}$.
\end{lemma}

The primary tool in our proofs is an extrapolation theorem appeared in \cite{LR3}.
\begin{lemma}\cite{LR3}\label{th:lem2}
Let $\nu$ be a weight and $T$ a sublinear operator defined in $C_{c}^{\infty}(\mathbb{R})$ and satisfying
$$\|\tau Tf\|_\infty\leq C\|\sigma f\|_\infty,$$
for all $\tau$ and $\sigma$ such that $\sigma=\nu\tau$, $\tau^{-1}\in A_1^{-}$ and $\sigma^{-1}\in A_1$. Then
for $1<p<\infty$,
$$\|Tf\|_{L^p(w)}\leq C\|f\|_{L^p(v)},$$
holds whenever $w\in A_p^{+}$ and $v=\nu^pw\in A_p$.
\end{lemma}

Based on Lemma 2.3 in \cite{LR3}, we get the following estimate which is essential to the proofs of Theorem 1.1 and 1.2.
\begin{lemma}\label{th:lem3}
Let $0<\alpha<1$, $\mu\in A_1$ and $b\in Lip_{\alpha,\mu}$.
Assume that $\tau$ and $\sigma=\mu^{1+\alpha}\tau$ are such that
$\tau^{-1}\in A_1^{-}$ and $\sigma^{-1}\in A_1$. Then there exists $\varepsilon_2>0$ such that for all $1<r<1+\varepsilon_2$,
$$\frac{1}{|I|^{\alpha}}\left(\frac{1}{|I|}\int_{I}|b(y)-b_I|^{r}\sigma^{-r}dy\right)^{1/r}
\leq C\|b\|_{Lip_{\alpha},\mu}\tau^{-1}(x),\quad a.e.~~ x\in\mathbb{R}. $$
where $I=[x,x+h]$.
\end{lemma}

\begin{proof}
Since $\tau^{-1}\in A_1^{-}$, by Lemma \ref{th:lem1}, there exists $\varepsilon_1>0$ such that for all
$1<r\leq 1+\varepsilon_1$,
$\tau^{-r}\in A_1^{-}$. By the fact that $\mu\in A_1$, we have
\begin{equation*}\begin{split}
&\frac{1}{|I|^{\alpha}}\left(\frac{1}{|I|}\int_{I}|b(y)-b_I|^{r}\sigma^{-r}(y)dy\right)^{1/r}\\
&\leq\frac{1}{|I|^{\alpha}}\left\{\frac{1}{|I|}\int_{I}\sup_{J\ni y\atop
J\subset I}
\left(\frac{1}{|J|}\int_J|b(t)-b_J|dt\right)^r\sigma^{-r}(y)dy\right\}^{1/r}\\
&\leq\frac{1}{|I|^{\alpha}}\left\{\frac{1}{|I|}\int_{I}\sup_{J\ni y\atop
J\subset I}\left(\frac{\mu(J)}{|J|}\right)^{(1+\alpha)r}|J|^{\alpha r}
\left(\frac{1}{\mu(J)^{1+\alpha}}\int_J|b(t)-b_J|dt\right)^r\sigma^{-r}(y)dy\right\}^{1/r}\\
&\leq C\|b\|_{Lip_{\alpha,\mu}}\frac{1}{|I|^{\alpha}}\left\{\frac{1}{|I|}\int_{I}|I|^{\alpha r}
\mu(y)^{(1+\alpha)r}\sigma^{-r}(y)dy\right\}^{1/r}\\
&\leq C\|b\|_{Lip_{\alpha,\mu}}\left\{\frac{1}{|I|}\int_{I}
\tau^{-r}(y)dy\right\}^{1/r}\\
&\leq C\|b\|_{Lip_{\alpha,\mu}}\tau^{-1}(x),
\end{split}\end{equation*}
for almost all $x\in\mathbb{R}$.
\end{proof}

\begin{lemma}\label{th:lem4}
Assume that $b\in Lip_{\alpha,\mu}$, $\mu\in A^1$, $x\in\mathbb{R}$ and $h>0$. For each $j\in\mathbb{Z}^+$, let
$I_j=[x,x+2^jh]$, $j\geq 3$. Then
$$\frac{1}{h^{\alpha}}|b_{I_{j+1}}-b_{I_3}|\leq C\|b\|_{Lip_{\alpha},\mu}
\frac{2^{4\alpha}(1-2^{(j-2)\alpha})}{1-2^{\alpha}}\mu(x)^{1+\alpha}.$$
\end{lemma}

\begin{proof}
Since $\mu\in A^1$, we have
\begin{equation*}\begin{split}
\frac{1}{h^{\alpha}}|b_{I_{m}}-b_{I_{m+1}}|
&\leq\frac{1}{h^{\alpha}}\frac{1}{|I_m|}\int_{I_m}|b(t)-b_{I_{m+1}}|dt\\
&\leq C 2^{(m+1)\alpha}\left(\frac{\mu(I_{m+1})}{|I_{m+1}|}\right)^{1+\alpha}\|b\|_{Lip_{\alpha},\mu}\\
&\leq C 2^{(m+1)\alpha}\|b\|_{Lip_{\alpha},\mu}\mu(x)^{1+\alpha}.
\end{split}\end{equation*}
Therefore,
\begin{equation*}\begin{split}
\frac{1}{h^{\alpha}}|b_{j+1}-b_{I_3}|&=\sum_{m=3}^{j}|b_{I_m}-b_{I_{m+1}}|\\
&\leq C\|b\|_{Lip_{\alpha},\mu}\mu(x)^{1+\alpha}\sum_{m=3}^{j}2^{(m+1)\alpha}\\
&\leq C\|b\|_{Lip_{\alpha},\mu}
\frac{2^{4\alpha}(1-2^{(j-2)\alpha})}{1-2^{\alpha}}\mu(x)^{1+\alpha}.
\end{split}\end{equation*}
The lemma is proved.
\end{proof}

Using some notations of \cite{LR1}, \cite{LR2} and \cite{LR3}, we will prove Theorem 1.1 and Theorem 1.2, respectively.

\section{Weighted estimates for commutators of one-sided singular integrals}

\begin{proof}[Proof of Theorem $\ref{th:1}$]
Let $\lambda$ be an arbitrary constant. Then
$$T_{b}^{+}f(x)= T^{+}((\lambda-b)f)(x)+(b(x)-\lambda)
T^{+}f(x).$$
Let $x\in\mathbb{R}$, $h>0$, $J=[x,x+8h]$. Write $f=f_1+f_2$, where $f_1=f\chi_J$, set $\lambda=b_J$.
Then
\begin{equation*}\begin{split}
&\frac{1}{h^{1+\alpha}}\int_{x}^{x+2h}|T^+_{b}f(y)-(T^+_{b}f)_{[x, x+2h]}|dy\\
&\leq\frac{2}{h^{1+\alpha}}\int_{x}^{x+2h}|T_{b}^{+}f(y)-T^{+}((b-b_{J})f_{2})(x+2h)|dy\\
&\leq\frac{2}{h^{1+\alpha}}\int_{x}^{x+2h}|T^{+}((b-b_{J})f_{1})(y)|dy \\
&+
\frac{2}{h^{1+\alpha}}\int_{x}^{x+2h}|T^{+}((b-b_{J})f_{2})(y)
-T^{+}((b-b_{J})f_{2})(x+2h)|dy\\
&+ \frac{2}{h^{1+\alpha}}\int_{x}^{x+2h}|b(y)-b_{J}||T^{+}f(y)|dy\\
&=2(I(x)+II(x)+III(x)).
\end{split}\end{equation*}
By definition of Calder\'on-Zygmund kernel, we have
$$II(x)\leq C\frac{1}{h^{1+\alpha}}\int_{x}^{x+2h}\int_{x+8h}^{\infty}\frac{x+2h-y}{(t-(x+2h))^{2}}|b(t)-b_{J}||f(t)|dtdy.$$

Consider the following three sublinear operators defined on $C_c^{\infty}$:
\begin{equation}\begin{split}
&M^{+}_{1}f(x)=\sup_{h>0}\frac{1}{h^{1+\alpha}}\int_{x}^{x+2h}|T^{+}((b-b_{J})f\chi_{J})(y)|dy,\\
&M^{+}_{2}f(x)=\sup_{h>0}\frac{1}{h^{1+\alpha}}\int_{x}^{x+2h}\int_{x+8h}^{\infty}\frac{x+2h-y}{(t-(x+2h))^{2}}|b(t)-b_{J}||f(t)|dtdy,
\\
&M^{+}_{3}g(x)=\sup_{h>0}\frac{1}{h^{1+\alpha}}\int_{x}^{x+2h}|b(y)-b_{[x,x+8h]}||g(y)|dy.\label{MMM}
\end{split}\end{equation}
The above inequalities imply that
\begin{equation}\begin{split}
&\frac{1}{h^{1+\alpha}}\int_{x}^{x+2h}|T_{b}^{+}f(y)-(T_{b}^{+}f)_{[x,
x+2h]}|dy\\
&\quad\leq
C\left(M^{+}_{1}f(x)+M^{+}_{2}f(x)+M^{+}_{3}(T^{+}f)(x)\right).\label{co}
\end{split}\end{equation}

Now, let's discuss the boundedness of these three operators. For $M^+_{1}$.
Assume that $\tau$ and $\sigma=\mu^{1+\alpha}\tau$ are such that $\tau^{-1}\in A^{-}_{1}$ and $\sigma^{-1}\in
A_{1}$. Let $1<r<1+\varepsilon_2$, where $\varepsilon_2$ is as in Lemma \ref{th:lem3}. By H\"older's inequality, Lemma 
\ref{th:lem3} and the fact that $T^{+}$ is bounded from
$L^{r}(\mathbb{R})$ to $L^{r}(\mathbb{R})$ \cite{AF}, we get
\begin{equation*}\begin{split}
&\frac{1}{h^{1+\alpha}}\int_{x}^{x+2h}|T^{+}((b-b_{J})f\chi_{J})(y)|dy\\
&\leq \frac{C}{h^{\alpha}}\left(\frac{1}{h}\int_{x}^{x+2h}|T^{+}((b-b_{J})f\chi_{J})(y)|^{r}dy\right)^{1/r}\\
&\leq
\frac{C}{h^{\alpha}}\left(\frac{1}{h}\int_{x}^{x+8h}|(b(y)-b_{J})f(y)|^{r}dy\right)^{1/r}\\
&\leq C\|f\sigma\|_{\infty}\frac{1}{h^{\alpha}}\left(\frac{1}{h}\int_{x}^{x+8h}|b(y)-b_J|^{r}\sigma^{-r}(y)dy\right)^{1/r}\\
&\leq
C\|b\|_{Lip_{\alpha},\mu}\|f\sigma\|_{\infty}\tau^{-1}(x).
\end{split}\end{equation*}
Therefore,
$$\|\tau M_1^{+}f\|_\infty\leq C\|f\sigma\|_{\infty}.$$
Then by Lemma \ref{th:lem2}, for $w\in A_p^{+}$ and $v=\mu^{(1+\alpha)p}w\in A_p$, we have
\begin{equation}
\|M_{1}^{+}f\|_{L^{p}(w)}\leq
C\|f\|_{L^{p}(v)}.\label{M1}
\end{equation}

For $M_2^{+}$, let $I_j=[x,x+2^jh]$, $j\in\mathbb{Z}^+$. Then
\begin{equation}\begin{split}
&
\frac{1}{h^{1+\alpha}}\int_{x}^{x+2h}\int_{x+8h}^{\infty}\frac{x+2h-y}{(t-(x+2h))^{2}}|b(t)-b_{J}||f(t)|dtdy\\
&\leq
\frac{C}{h^{\alpha}}\int_{x}^{x+2h}\sum_{j=3}^{\infty}\int_{x+2^{j}h}^{x+2^{j+1}h}
\frac{|b(t)-b_{J}|}{(t-(x+2h))^{2}}|f(t)|dtdy\\
&\leq C
\sum_{j=3}^{\infty}\frac{1}{(2^{j}-2)^{2}h^{1+\alpha}}\int_{x+2^{j}h}^{x+2^{j+1}h}
|b(t)-b_{J}||f(t)|dt\\
&\leq C
\sum_{j=3}^{\infty}\frac{2^{j+1}}{(2^{j}-2)^{2}}\Biggl(\frac{1}{2^{j+1}h^{1+\alpha}}\int_{I_{j+1}}
|b(t)-b_{I_{j+1}}||f(t)|dt\\
&\qquad+\dfrac{1}{2^{j+1}h^{1+\alpha}}\int_{I_{j+1}}
|b_{I_{j+1}}-b_{J}||f(t)|dt\Biggr)\\
&= C
\sum_{j=3}^{\infty}\frac{1}{2^{j}}\big(II_{1}(x)+II_{2}(x)\big).\label{M21}
\end{split}\end{equation}

By H\"{o}lder's inequality and Lemma \ref{th:lem3}, we have

\begin{equation}\begin{split}
II_{1}(x)&=\frac{1}{2^{j+1}h^{1+\alpha}}\int_{I_{j+1}}
|b(t)-b_{I_{j+1}}||f(t)|dt\\
&\leq
\frac{1}{h^{\alpha}}\left(\frac{1}{2^{j+1}h}
\int_{I_{j+1}}|b(t)-b_{I_{j+1}}|^r|f|^{r}dt\right)^{1/r}
\\&\leq
\|f\sigma\|_{\infty}\frac{1}{h^{\alpha}}\left(\frac{1}{2^{j+1}h}
\int_{I_{j+1}}|b(t)-b_{I_{j+1}}|^r\sigma^{-r}dt\right)^{1/r}\\
&\leq
C2^{(j+1)\alpha}\|b\|_{Lip_{\alpha,\mu}}\|f\sigma\|_{\infty}\tau^{-1}(x).\label{2-1}
\end{split}\end{equation}

Since $\sigma^{-1}\in A_1$, then by Lemma \ref{th:lem4},
\begin{equation}\begin{split}
II_{2}(x)&=\frac{1}{2^{j+1}h^{1+\alpha}}\int_{I_{j+1}}
|b_{I_{j+1}}-b_{J}||f(t)|dt\\
&\leq \frac{1}{h^\alpha}|b_{I_{j+1}}-b_{J}|\|f\sigma\|_{\infty}
\frac{1}{|I_{j+1}|}\int_{I_{j+1}}
\sigma^{-1}dt\\
&\leq C\|b\|_{Lip_{\alpha},\mu}
\frac{2^{4\alpha}(1-2^{(j-2)\alpha})}{1-2^{\alpha}}\mu(x)^{1+\alpha}
\|f\sigma\|_{\infty}\sigma^{-1}(x)\\
&= C\|b\|_{Lip_{\alpha},\mu}
\frac{2^{4\alpha}(1-2^{(j-2)\alpha})}{1-2^{\alpha}}
\|f\sigma\|_{\infty}\tau^{-1}(x).\label{2-2}
\end{split}\end{equation}
Then \eqref{M21}-\eqref{2-2} indicate that
\begin{equation}\begin{split}
&
\frac{1}{h^{1+\alpha}}\int_{x}^{x+2h}\int_{x+8h}^{\infty}\frac{x+2h-y}{(t-(x+2h))^{2}}|b(t)-b_{J}||f(t)|dtdy\\
&\leq
C\|b\|_{Lip_{\alpha,\mu}}\|f\sigma\|_{\infty}\tau^{-1}(x)
\sum_{j=3}^{\infty}\frac{1}{2^{j}}\left(2^{(j+1)\alpha}
+
\frac{2^{4\alpha}(1-2^{(j-2)\alpha})}{1-2^{\alpha}}
\right)\\
&\leq C\|b\|_{Lip_{\alpha,\mu}}\|f\sigma\|_{\infty}\tau^{-1}(x),
\end{split}\end{equation}
where the last inequality is due to the fact that $0<\alpha<1$. Consequently,
$$\|\tau M_2^{+}f\|_\infty\leq C\|f\sigma\|_{\infty}.$$
Then by Lemma \ref{th:lem2}, for $w\in A_p^{+}$ and $v=\mu^{(1+\alpha)p}w\in A_p$, we have
\begin{equation}
\|M_{2}^{+}f\|_{L^{p}(w)}\leq
C\|f\|_{L^{p}(v)}.\label{M2}
\end{equation}

For $M^+_{3}$.
 By H\"older's inequality and Lemma \ref{th:lem4}, we get
\begin{equation*}\begin{split}
&\frac{1}{h^{1+\alpha}}\int_{x}^{x+2h}|b(y)-b_{J}||g(y)|dy\\
&\leq \frac{C}{h^{\alpha}}\left(\frac{1}{h}\int_{x}^{x+2h}|b(y)-b_{J}|^{r}|g(y)|^{r}dy\right)^{1/r}\\
&\leq C\|g\sigma\|_{\infty}\frac{1}{h^{\alpha}}\left(\frac{1}{h}\int_{x}^{x+8h}|b(y)-b_J|^{r}\sigma^{-r}(y)dy\right)^{1/r}\\
&=
C\|b\|_{Lip_{\alpha},\mu}\|g\sigma\|_{\infty}\tau^{-1}(x).
\end{split}\end{equation*}
Thus,
$$\|\tau M_3^{+}g\|_\infty\leq C\|g\sigma\|_{\infty}.$$
From Lemma \ref{th:lem2}, we get
\begin{equation}
\|M_{3}^{+}g\|_{L^{p}(w)}\leq
C\|g\|_{L^{p}(v)},\label{M31}
\end{equation}
where $w\in A_p^{+}$ and $v=\mu^{(1+\alpha)p}w\in A_p$.
Since $T^+$ is bounded from $L^p(v)$ to $L^p(v)$ \cite{AF}, it follows that
\begin{equation}
\|M_{3}^{+}(T^+f)\|_{L^{p}(w)}\leq
C\|T^+f\|_{L^{p}(v)}\leq
C\|f\|_{L^{p}(v)}. \label{M3}
\end{equation}

Consequently, by \eqref{co}, \eqref{M1}, \eqref{M2} and \eqref{M3}, we obtain
$$\|T_b^+f\|_{\dot{F}^{\alpha,
\infty}_{p,
+}(\omega)}\approx\left\|\sup_{h>0}\frac{1}{h^{1+\alpha}}\int_{x}^{x+h}|T_b^+f-(T_b^+f)_{[x,
x+h]}|\right\|_{L^{p}(\omega)}\leq
C\|f\|_{L^{p}(v)}.$$
This completes the proof of Theorem \ref{th:1}.
\end{proof}

\section{Weighted estimates for commutators of one-sided discrete square functions}

\begin{proof}[Proof of Theorem $\ref{th:2}$]
The procedure of this proof is analogous to that of Theorem \ref{th:1}.
Let $\lambda$ be an arbitrary constant. Then
\begin{equation*}\begin{split}
S_{b}^{+}f(x)&= \left\|\int_{\mathbb{R}}(b(x)-b(y))H(x-y)f(y)dy\right\|_{l^2}\\
&\leq\left\|(b(x)-\lambda)\int_{\mathbb{R}}H(x-y)f(y)dy\right\|_{l^2}+
\left\|\int_{\mathbb{R}}H(x-y)(b(y)-\lambda)f(y)dy\right\|_{l^2}\\
&=|b(x)-\lambda|S^{+}f(x)+S^{+}((b-\lambda)f)(x).
\end{split}\end{equation*}
Let $x\in\mathbb{R}$, $h>0$ and let $j\in\mathbb{Z}$ be such that $2^j\leq h<2^{j+1}$.
Set $J=[x,x+2^{j+3}]$. Write $f=f_1+f_2$, where $f_1=f\chi_J$, set $\lambda=b_J$.
Then
\begin{equation*}\begin{split}
&\frac{1}{h^{1+\alpha}}\int_{x}^{x+2h}|S^+_{b}f(y)-(S^+_{b}f)_{[x, x+2h]}|dy\\
&\leq\frac{2}{h^{1+\alpha}}\int_{x}^{x+2h}|S_{b}^{+}f(y)-S^{+}((b-b_{J})f_{2})(x)|dy\\
&\leq\frac{2}{h^{1+\alpha}}\int_{x}^{x+2h}|S^{+}((b-b_{J})f_{1})(y)|dy \\
&+
\frac{2}{h^{1+\alpha}}\int_{x}^{x+2h}|S^{+}((b-b_{J})f_{2})(y)
-S^{+}((b-b_{J})f_{2})(x)|dy\\
&+ \frac{2}{h^{1+\alpha}}\int_{x}^{x+2h}|b(y)-b_{J}||S^{+}f(y)|dy\\
&=2(L(x)+LL(x)+LLL(x)).
\end{split}\end{equation*}
By definition, we have
\begin{equation*}\begin{split}
LL(x)&\leq \frac{1}{h^{1+\alpha}}\int_{x}^{x+2^{j+2}}\|U^+((b-b_J)f_2)(y)-U^+((b-b_J)f_2)(x)\|_{l^2}\\
&\leq \frac{1}{h^{1+\alpha}}\int_{x}^{x+2^{j+2}}\int_{x+2^{j+3}}^{\infty}|(b(t)-b_J)f(t)|\|H(y-t)-H(x-t)\|_{l^2} dtdy.
\end{split}\end{equation*}

Define sublinear operators:
\begin{equation*}\begin{split}
&M^{+}_{4}f(x)=\sup_{j\in\mathbb{Z}}\frac{1}{2^{j(1+\alpha)}}\int_{x}^{x+2^{j+2}}|S^{+}((b-b_{J})f\chi_{J})(y)|dy,\\
&M^{+}_{5}f(x)=\sup_{j\in\mathbb{Z}}\frac{1}{2^{j(1+\alpha)}}\int_{x}^{x+2^{j+2}}
\int_{x+2^{j+3}}^{\infty}|(b(t)-b_J)f(t)|\|H(y-t)-H(x-t)\|_{l^2} dtdy.\\
\end{split}\end{equation*}
It follows that
\begin{equation}\begin{split}
&\frac{1}{h^{1+\alpha}}\int_{x}^{x+2h}|S^+_{b}f(y)-(S^+_{b}f)_{[x, x+2h]}|dy\\
&\quad\leq
C\left(M^{+}_{4}f(x)+M^{+}_{5}f(x)+M^{+}_{3}(S^{+}f)(x)\right),\label{co1}
\end{split}\end{equation}
where $M^{+}_{3}$ is defined in \eqref{MMM}. It follows from \eqref{M3} that
$$ \|M_{3}^{+}(S^+f)\|_{L^{p}(w)}\leq
C\|S^+f\|_{L^{p}(v)}.$$

 By Theorem A in \cite{TT}, we have
$$ \|S^+f\|_{L^{p}(v)}\leq
C\|f\|_{L^{p}(v)}.$$
Therefore,
\begin{equation}
 \|M_{3}^{+}(S^+f)\|_{L^{p}(w)}\leq
C\|f\|_{L^{p}(v)}\label{M33}
\end{equation}
holds for $w\in A_p^{+}$ and $v=\mu^{(1+\alpha)p}w\in A_p$.

Next we shall prove that $M^+_{4}$, $M^+_{5}$ are all bounded from $L^{p}(v)$ to $L^{p}(w)$.
 For $M^+_{4}$.
Assume that $\tau$ and $\sigma=\mu^{1+\alpha}\tau$ are such that $\tau^{-1}\in A^{-1}_{1}$ and $\sigma^{-1}\in
A_{1}$. By H\"older's inequality, Lemma \ref{th:lem3} and the fact that $S^{+}$ is bounded from
$L^{r}(\mathbb{R})$ to $L^{r}(\mathbb{R})$ \cite{TT}, we get
\begin{equation*}\begin{split}
&\frac{1}{2^{j(1+\alpha)}}\int_{x}^{x+2^{j+2}}|S^{+}((b-b_{J})f\chi_{J})(y)|dy\\
&\leq \frac{C}{2^{j\alpha}}\left(\frac{1}{2^{j}}\int_{x}^{x+2^{j+2}}|S^{+}((b-b_{J})f\chi_{J})(y)|^{r}dy\right)^{1/r}\\
&\leq \frac{C}{2^{j\alpha}}\left(\frac{1}{2^{j}}\int_{x}^{x+2^{j+3}}|(b(y)-b_{J})f(y)|^{r}dy\right)^{1/r}\\
&\leq C\|f\sigma\|_{\infty}\frac{1}{2^{j\alpha}}\left(\frac{1}{2^{j}}\int_{x}^{x+2^{j+3}}|b(y)-b_J|^{r}\sigma^{-r}(y)dy\right)^{1/r}\\
&=
C\|b\|_{Lip_{\alpha},\mu}\|f\sigma\|_{\infty}\tau^{-1}(x).
\end{split}\end{equation*}
Therefore,
$$\|\tau M_4^{+}f\|_\infty\leq C\|f\sigma\|_{\infty}.$$
Then by Lemma \ref{th:lem2}, the inequality
\begin{equation}
\|M_{4}^{+}f\|_{L^{p}(w)}\leq
C\|f\|_{L^{p}(v)}\label{M4}
\end{equation}
holds for $w\in A_p^{+}$ and $v=\mu^{(1+\alpha)p}w\in A_p$.

For $M_5^{+}$, let $I_j=[x,x+2^{j}]$, $j\in\mathbb{Z}$. Then
\begin{equation}\begin{split}
&
\int_{x+2^{j+3}}^{\infty}|(b(t)-b_J)f(t)|\|H(y-t)-H(x-t)\|_{l^2} dt\\
&\leq
\sum_{k=j+3}^{\infty}\int_{x+2^{k}}^{x+2^{k+1}}
|(b(t)-b_{I_{k+1}})f(t)|\|H(y-t)-H(x-t)\|_{l^2}dt\\
&+
\sum_{k=j+3}^{\infty}|b_{I_{k+1}}-b_J|\int_{x+2^{k}}^{x+2^{k+1}}
|f(t)|\|H(y-t)-H(x-t)\|_{l^2}dt\\
&= LL_{1}(x)+LL_{2}(x).\label{M511}
\end{split}\end{equation}

Since $\tau$ and $\sigma=\mu^{1+\alpha}\tau
=(\frac{v}{w})^{\frac{1}{p}}\tau$
are such that $\tau^{-1}\in A^{-}_{1}$ and $\sigma^{-1}\in
A_{1}\subset A^{-}_{1}$, by Lemma \ref{th:lem1}, there exists $\varepsilon>0$ such that when $1<r<1+\varepsilon$, 
$\tau^{-r}\in A^{-}_{1}$ and $\sigma^{-r}\in A^{-}_{1}$. Since $\alpha<1-\frac{1}{1+\varepsilon}$,
we can choose $r>1$ such that $\alpha<\frac{1}{r'}$,
then by H\"older's inequality and Lemma \ref{th:lem3}, 
\begin{equation*}\begin{split}
LL_{1}(x)\leq C\sum_{k=j+3}^{\infty}&\left(\int_{I_{k+1}}|b(t)-b_{I_{k+1}}|^r
|f(t)|^rdt\right)^{\frac{1}{r}}\times\\
& \left(\int_{x+2^{k}}^{x+2^{k+1}}
\|H(y-t)-H(x-t)\|^{r'}_{l^2}dt\right)^{\frac{1}{r'}}.\\
\end{split}\end{equation*}

By Theorem 1.6 in \cite{TT}, for all $y\in[x,x+2^{j+3}]$, the kernel
$H$ satisfies
\begin{equation}
\left(\int_{x+2^{k}}^{x+2^{k+1}}
\|H(y-t)-H(x-t)\|^{r'}_{l^2}dt\right)^{\frac{1}{r'}}\leq C\frac{2^{\frac{j}{r'}}}{2^{k}}.\label{H}
\end{equation}
Therefore
\begin{equation}\begin{split}
LL_{1}(x)&\leq
C\|f\sigma\|_\infty
\sum_{k=j+3}^{\infty}\frac{2^{\frac{j}{r'}}}{2^{k}}\left(\int_{I_{k+1}}|b(t)-b_{I_{k+1}}|^r\sigma^{-r}(t)dt
\right)^{\frac{1}{r}}
\\&\leq C
\|f\sigma\|_\infty\|b\|_{Lip_{\alpha},\mu}\tau^{-1}(x)
\sum_{k=j+3}^{\infty}\frac{2^{\frac{j}{r'}}}{2^{k}}|I_{k+1}|^{\alpha+\frac{1}{r}}
\\
&\leq
C\|f\sigma\|_\infty\|b\|_{Lip_{\alpha},\mu}\tau^{-1}(x)
\sum_{k=j+3}^{\infty}\frac{2^{\frac{j}{r'}}}{2^{k}}2^{(k+1)(\alpha+\frac{1}{r})}
\\
&\leq C 2^{j\alpha}\|b\|_{Lip_{\alpha},\mu}\|f\sigma\|_\infty\tau^{-1}(x). \label{LL1}
\end{split}\end{equation}

By the same proof as in Lemma \ref{th:lem4} we can get that
$$|b_{I_{k+1}}-b_J|=\sum_{m=j+3}^{k}|b_{I_{m+1}}-b_m|
\leq C(2^{j\alpha}+2^{k\alpha})\|b\|_{Lip_{\alpha},\mu}\mu^{1+\alpha}(x) $$

Then by \eqref{H}, H\"older's inequality and the fact that $\sigma^{-r}\in A^{-}_1$, $\alpha<\frac{1}{r'}$, we have
\begin{equation}\begin{split}
LL_{2}(x)&\leq C\|b\|_{Lip_{\alpha},\mu}\mu^{1+\alpha}(x)\sum_{k=j+3}^{\infty}\frac{2^{\frac{j}{r'}}(2^{j\alpha}+2^{k\alpha})}{2^{k}}
\left(\int_{I_{k+1}}|f(t)|^rdt\right)^{\frac{1}{r}}
\\
&\leq C\|b\|_{Lip_{\alpha},\mu}\|f\sigma\|_\infty\mu^{1+\alpha}(x)\\
&\times\sum_{k=j+3}^{\infty}
\frac{2^{\frac{j}{r'}}(2^{j\alpha}+2^{k\alpha})}{2^{k}}\left(\int_{I_{k+1}}\sigma^{-r}(t)dt\right)^{\frac{1}{r}}
\\
&\leq C2^{j\alpha}\|b\|_{Lip_{\alpha},\mu}\|f\sigma\|_\infty\tau^{-1}(x).\label{L2-2}
\end{split}\end{equation}
Following from \eqref{M511}, \eqref{LL1} and \eqref{L2-2}, we get
\begin{equation*}\begin{split}
\int_{x+2^{j+3}}^{\infty}|(b(t)-b_J)f(t)|\|H(y-t)-H(x-t)\|_{l^2} dt
\leq
C2^{j\alpha}\|b\|_{Lip_{\alpha},\mu}\|f\sigma\|_\infty\tau^{-1}(x).
\end{split}\end{equation*}
Consequently,
\begin{equation*}\begin{split}
&\frac{1}{2^{j(1+\alpha)}}\int_{x}^{x+2^{j+2}}
\int_{x+2^{j+3}}^{\infty}|(b(t)-b_J)f(t)|\|H(y-t)-H(x-t)\|_{l^2} dtdy\\
&\leq C\|b\|_{Lip_{\alpha},\mu}\|f\sigma\|_\infty\tau^{-1}(x).
\end{split}\end{equation*}
Therefore,
$$\|\tau M_5^{+}f\|_\infty\leq C\|f\sigma\|_{\infty}.$$
Then by Lemma \ref{th:lem2}, the inequality
\begin{equation}
\|M_{5}^{+}f\|_{L^{p}(w)}\leq
C\|f\|_{L^{p}(v)}\label{M5}
\end{equation}
holds for $w\in A_p^{+}$ and $v=\mu^{(1+\alpha)p}w$.
Then Theorem \ref{th:2} follows form \eqref{co1}-\eqref{M4} and \eqref{M5}.
\end{proof}

\begin{remark}
It should be noted that by the well-known extrapolation theorem
appeared in \cite{ST} and the similar estimate of Lemma
\ref{th:lem3}, we can also obtain the corresponding boundedness for
commutators generated by `both-sided' singular integrals and
weighted Lipschitz functions from weighted Lebesgue spaces to
weighted `both-sided' Trieble-Lizorkin spaces. In brief, we leave
the completion of the proof to the interested readers.

\end{remark}
\medskip

\bibliographystyle{amsplain}

\end{document}